\documentclass[12pt]{article}
\usepackage[left=3cm,right=3cm, top=2.5cm,bottom=2.5cm,bindingoffset=0cm]{geometry}
\usepackage{tikz}
\usepackage{subfigure}
\usepackage{amsmath}
\usepackage{amssymb}
\usepackage{amsthm}
\usepackage{amsfonts}
\usepackage{euscript}
\usepackage{cite}
\usepackage{caption}
\usepackage{authblk}
\usepackage{etoolbox}

\patchcmd{\titlepage}
  {\thispagestyle{empty}}
  {\thispagestyle{plain}}
  {}
  {}

\title{Morse-Darboux lemma for surfaces with boundary}
\author{Ilia Kirillov\thanks{
e-mail: {\tt ikirillov@abc.math.msu.su}
}}
\affil{Lomonosov Moscow State University}
\date{}

\theoremstyle{plain}
\newtheoremstyle{dotless}{}{}{\itshape}{}{\bfseries}{}{ }{}
\newtheorem{theorem}{Theorem}
\newtheorem{Lemma}{Lemma}
\newtheorem{Con}{Corollary}
\newtheorem{Prep}{Proposition}
\newtheorem{Ex}{Example}

\theoremstyle{remark}
\newtheorem{Rem}{Remark}

\theoremstyle{definition}
\newtheorem{Def}{Definition}

\theoremstyle{dotless}
\newtheorem*{thm}{Theorem}

\begin{document}
\maketitle
\pagestyle{plain}

\begin{abstract}
	We formulate and prove an analog of the classical Morse-Darboux lemma for the case of a surface with boundary.
\end{abstract}


\section{Introduction}
Throughout this paper the word \textit{smooth} means $C^{\infty}$ smooth. The aim of this paper is to prove the following theorem.
\begin{theorem}
Let $M$ be a 2D surface with an area form $\omega$, and let $f:M\to\mathbb R$ be a smooth function. Let also $O\in \partial M$  be a regular point for $f$ and a non-degenerate critical point for $f|_{\partial M}.$ Then there exists a chart $(p,q)$ centered at $O$ such that we have $q \geq 0$ wherever q is defined, the boundary $\partial M$ satisfies the equation $q=0,$ $\omega=dp\wedge dq,$  and $f=\alpha \circ S$, where $S=q+p^2$ or $S=q-p^2$ (See Figure~\ref{fig1}). The function $\alpha$ of one variable is smooth in the neighborhood of the origin $0 \in \mathbb R$ and $\alpha'(0) \ne 0.$ 
\end{theorem}

\bigskip

Theorem~$1$ is closely related to the classical Morse-Darboux lemma. Let us recall the statement of that lemma.
\begin{theorem}
Let $M$ be a 2D surface with an area form $\omega$, and let $f:M\to\mathbb R$ be a smooth function. Let also $O\in M\setminus \partial M$  be a non-degenerate critical point
for $f.$
Then there exists a chart $(p,q)$ centered at $O$ such that $\omega=dp\wedge dq,$ and $f=\alpha \circ S$, where $S=pq$ or $S=p^2+q^2$. The function $\alpha$ of one variable is smooth in the neighborhood of the origin $0 \in \mathbb R$ and $\alpha'(0) \ne 0.$
\end{theorem}

\begin{figure}
\centering
\subfigure[Case $S=q+p^2.$]{
	\begin{tikzpicture}[scale=2.5]
	\filldraw (0,0) circle (0.2pt);
	\clip (-1,-0.2) rectangle (1.2,1.2);
	\draw[->, thick] (-1,0) -- (1,0) node[right] {$p$};
	\draw[->] (0,0) -- (0,1) node[above] {$q$};
	\foreach \x in {0.4, 0.3, 0.2, 0.1, 0.0}
		\draw  (-1,\x+1) parabola bend (0,\x) (1,\x+1);
	\foreach \x in {0.4, 0.3, 0.2, 0.1}
		\draw[domain={sqrt(\x)}:1, smooth, variable=\y] plot ({\y},{\y*\y-\x});
	\foreach \x in {0.4, 0.3, 0.2, 0.1}
		\draw[domain=-1:{-sqrt(\x)}, smooth, variable=\z] plot ({\z},{\z*\z-\x});
	\end{tikzpicture}}
\subfigure[Case $S=q-p^2.$]{
	\begin{tikzpicture}[scale=2.5]
	\filldraw (0,0) circle (0.2pt);
	\clip (-1,-0.2) rectangle (1.2,1.2);
	\draw[->, thick] (-1,0) -- (1,0) node[right] {$p$};
	\draw[->] (0,0) -- (0,1) node[above] {$q$};
	\foreach \x in {0.0, 0.1, 0.2, 0.3, 0.4, 0.5, 0.6, 0.7, 0.8, 0.9}
		\draw[domain={-sqrt(\x)}:{sqrt(\x)}, smooth, variable=\y] plot ({\y},{\x-\y*\y});
	\end{tikzpicture}}
\caption{Level sets of $f$. The horizontal axis is the boundary of $M.$}
\label{fig1}
\end{figure}
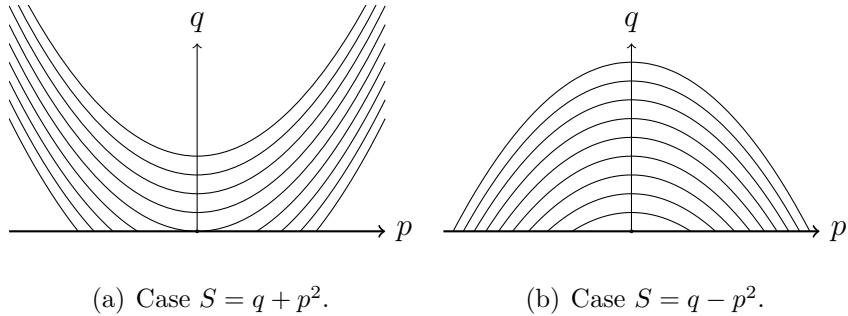

The Morse-Darboux lemma is a particular case of Le lemme de Morse isochore, see \cite{C}, and also a particular case of Eliasson's theorem on the normal form for  an integrable Hamiltonian system near a non-degenerate critical point, see \cite{E,B}. The Morse-Darboux lemma is an important tool in topological hydrodynamics, see \cite{I}, and theory of integrable systems, see \cite{D}.

We expect that the result of the present paper will also be useful in 2D fluid dynamics. In particular, it gives a partial answer to Problem $5.6$  from \cite{Iz} on the asymptotical properties of measures on Reeb graphs.

This paper is organised as follows. In Section $2$ we formulate Theorem~$1'$ which is equivalent to Theorem~$1.$ The proof of Theorem~$1'$ is given in Section~$4$. Section~$3$ contains several lemmas useful for the proof of Theorem~$1'.$ 

\section{Reformulation of the main theorem}
 
\begin{thm}$\boldsymbol{1'.}$
Let $\omega=\omega(x,y)dx\wedge dy$ be an area form on $\mathbb R^2,$ and $f=f(x,y)$ be a smooth function such that $f_{x}(0,0)=0$, $f_{y}(0,0)>0$ and $f_{xx}(0,0)>0.$
Then there exists a chart $(p,q)$ centered at $(0,0)$ such that $\omega =dp\wedge dq,$ $f(p,q)=\alpha(p^2+q),$ and $q=0$ if and only if $y=0$. The function $\alpha$ of one variable is smooth in the neighborhood of the origin $0 \in \mathbb R$ and $\alpha'(0) > 0.$ 
\end{thm}

\begin{Prep}
Theorem~$1$ follows from Theorem~$1'.$
\end{Prep}

\begin{proof}
Let us choose a chart $(x,y)$ centered at $O$ in $\partial M$ such that $P \in \partial M$ if and only if $y(P)=0.$ The function $f(x,y)$ and the form $\omega(x,y)dx\wedge dy$ 
can be smoothly extended on some neighborhood of $(0,0)$. As $(0,0)$ is non-degenerate critical point for $f|_{\partial M}$ we have $f_{x}(0,0)=0$, $f_{y}(0,0)\ne0$, $f_{xx}(0,0)\ne0.$ 
To fulfil conditions $f_{y}(0,0)>0$, $f_{xx}(0,0)>0,$ we may need some of the following transformations:
$f \rightarrow -f,$ $y \rightarrow -y.$   
Now, we obtain the chart $(p,q)$ from Theorem $1'.$ If $q \leq 0$ we need one more transformation: $q \rightarrow -q, p \rightarrow -p.$
It remains to resctrict the chart $(p,q)$ to the upper half plane. 
\end{proof}

\section{Necessary lemmas}

In this section we assume that conditions of Theorem~$1'$ hold. Also from now on we will assume that $f(0,0)=0.$ This will simplify notation.

First of all, we want to prove an analog of the classical Morse Lemma for a surface with boundary.

\begin{Lemma}
There exists a chart $(\hat{x},\hat{y})$ centered at $(0,0)$ such that
\begin{enumerate}
\item 
 $\hat{x}(0,0)=\hat{y}(0,0)=0$;
\item 
 $f(\hat{x},\hat{y})=\hat{x}^2+\hat{y}$;
\item 
 $\hat{y}(x,y)=0$ if and only if $y=0.$
\end{enumerate}
\end{Lemma} 

\begin{proof}
Hadamard's lemma implies that $$f(x,y)=f_1(x,y)x+f_2(x,y)y,$$ where $f_1$ and $f_2$ are smooth functions, and $f_1(0,0)=f_{x}(0,0)$, $f_2(0,0)=f_{y}(0,0).$
Since  $f_{x}(0,0)=0$ Hadamard's lemma similarly implies that

\begin{equation*}
    \begin{split}
 f(x,y)=(f_{11}(x,y)x+f_{12}(x,y)y)x+f_2(x,y)y\\
 =f_{11}(x,y)x^2+f_{12}(x,y)xy+f_2(x,y)y\\
 =(x\sqrt{f_{11}(x,y)})^2+y(f_{12}(x,y)x+f_2(x,y)).
    \end{split}
\end{equation*}
Recall that $f_{xx}(0,0)>0$ and also notice that $f_{11}(0,0)=\frac{1}{2}f_{xx}(0,0)$. Consider the following transformation of coordinates
\begin{align*}
    \hat{x}(x,y)&=\sqrt{f_{11}(x,y)}x\\
    \hat{y}(x,y)&=y(f_{12}(x,y)x+f_2(x,y)).
\end{align*}

The Jacobian determinant of this transformation at the point $(0,0)$ is equal to $\sqrt{f_{11}(0,0)}f_2(0,0)>0.$ It follows from the inverse function theorem that functions $\hat{x}$ and $\hat{y}$
form a chart centered at $(0,0).$  By construction
$$f(\hat{x},\hat{y})=\hat{x}^2+\hat{y},$$ and $\hat{y}(x,y)=y(f_{12}(x,y)x+f_2(x,y))=0$ if and only if $y=0.$ 
\end{proof}

\begin{Rem}
	It follows from Lemma~$1$ that without loss of generality it can be assumed in Theorem~$1'$ that in the chart $(x,y)$  we have $f(x,y) = x^2+y.$ So from now on we will forget about the chart~$(\hat{x},\hat{y}).$
\end{Rem}

\begin{Con}
Let $$D(f,\varepsilon) := \{ (x,y)\in \mathbb R^2 \mid f(x,y)\leq \varepsilon \mbox{ and } y\geq 0 \}. $$ 

Than the function $$A_{f}(\varepsilon) := \int\limits_{D(f,\varepsilon)}\omega(x,y)dx\wedge dy$$ is well-defined if  $\varepsilon\geqslant0$ is small enough (to use Lemma 1). Using the chart $(x,y)$  the function $A_f(\varepsilon)$ can be expressed as 
$$A_f(\varepsilon)=\int\limits_{-\sqrt{\varepsilon}}^{\sqrt{\varepsilon}}dx\int\limits_0^{\varepsilon-x^2}\omega(x,y)\\dy.$$
\end{Con}

\begin{Rem}
	The function $A_f$ gives us an invariant of a pair $(f,\omega).$ It will play a crucial role in the proof of Theorem~$1'.$
\end{Rem}                                                                   

\begin{Ex}
Consider the upper half-plane $H$ with an area form $\omega = dp \wedge dq$ and  a function $f = \alpha(p^2 + q),$ where $\alpha'(0)>0.$ Then the function $A_{f}$ can be expressed as 
\begin{multline*} 
A_{f}(\alpha(\varepsilon))= \int\limits_{D(\alpha(p^2+q),\alpha(\varepsilon))}dp\wedge dq \\
= \int\limits_{D(p^2+q,\varepsilon)}dp\wedge dq  = \int\limits_{-\sqrt{\varepsilon}}^{\sqrt{\varepsilon}}dp\int\limits_0^{\varepsilon-p^2}dq=\int\limits_{-\sqrt{\varepsilon}}^{\sqrt{\varepsilon}}(\varepsilon-p^2)dx=\frac{4}{3}\varepsilon\sqrt{\varepsilon}=\frac{4}{3}\varepsilon^{3/2},
\end{multline*}
so $$[A_f\circ \alpha](\varepsilon)=4/3\varepsilon^{3/2}$$
and $$\alpha(\varepsilon)=A_f^{-1}(4/3\varepsilon^{3/2})$$
or $$\alpha^{-1}(\varepsilon)=[3/4A_f(\varepsilon)]^{2/3}.$$
So we know how to determine the function $\alpha$ from Theorem~$1'.$ Now we want to prove that $\alpha$ is a smooth function.
\end{Ex}

\begin{Lemma}
The function $\tilde A (\varepsilon):= A_f(\varepsilon)^{2/3}$ 
is smooth in some neighborhood of zero.
\end{Lemma}

\begin{proof}
Let $$u(x,\varepsilon) := \int\limits_0^{\varepsilon-x^2}\omega(x,y)\\dy.$$
Note that $u$ is a smooth function of two variables. Further,
$$A_f(\varepsilon) = \int\limits_{-\sqrt{\varepsilon}}^{\sqrt{\varepsilon}}u(x,\varepsilon)dx.$$ 
Introducing a new variable $\delta=\sqrt{\varepsilon}$ we obtain
$$A_f(\delta)=\int\limits_{-\delta}^{\delta}u(x,\delta^2)dx.$$
This function is smooth and odd. Let us find the third order Taylor polynomial of $A_f(\delta):$ 
\begin{multline*}
A_f(\delta) = \int\limits_{-\delta}^{\delta}dx\int\limits_0^{\delta^2-x^2}\omega(x,y)dy = \int\limits_{-\delta}^{\delta}dx\int\limits_0^{\delta^2-x^2}[\omega(0,0)+O(x)+O(y)]dy\\ 
= \omega(0,0)  \int\limits_{-\delta}^{\delta}dx\int\limits_0^{\delta^2-x^2}dy +\int\limits_{-\delta}^{\delta}dx \int\limits_0^{\delta^2-x^2}O(x)dy + \int\limits_{-\delta}^{\delta}dx\int\limits_0^{\delta^2-x^2}O(y)dy\\
= \omega(0,0)\frac{4}{3}(\delta^2)^{3/2} + O(\int\limits_{-\delta}^{\delta}dx \int\limits_0^{\delta^2-x^2}x dy)+O(\int\limits_{-\delta}^{\delta}dx \int\limits_0^{\delta^2-x^2}y dy)\\
= \omega(0,0)\frac{4}{3}\delta^3 + O(\delta^4)+O(\delta^4)=  \omega(0,0)\frac{4}{3}\delta^3 +O(\delta^4).
\end{multline*}

It means, that $A_f(\delta)=\delta^3B(\delta)$, where the function $B(\delta)$ is smooth, even, and $B(0)\ne0$.
So, $A_f(\varepsilon)=\varepsilon^{3/2}B(\sqrt{\varepsilon})$ and $\tilde A (\varepsilon) = \varepsilon [B(\sqrt{\varepsilon})]^{2/3}.$
\end{proof}

\begin{Rem}
The function $\tilde A$ is defined only if $\varepsilon\geqslant0$. But it extends to a smooth function on a neighborhood of zero.
\end{Rem}

Further in this section we will try to do things in the same way as in the proof of the classical Darboux Lemma (see \cite{Ar}, p. 230).

\begin{Def}
	Recall that one-forms on a surface $M$ with a fixed area form $\omega$ may be identified with vector fields, and every smooth function 
	\mbox{$f: M \to \mathbb R$} determines a unique vector field $X_f,$ called the \textit{Hamiltonian vector field} with the \textit{Hamiltonian} $f,$ by requiring that for every
	vector field $Y$ on $M$ the identity $df(Y)=\omega(Y,X_f)$ holds. 
	Let also $P_f$ be the flow (\textit{hamiltonian flow}) corresponding to the vector field  $X_f.$ 
\end{Def}

\begin{Def}
	Recall that in the chart $(x,y)$ we have $f(x,y)=x^2+y$ (see Remark~$1$). 
	Let $t_f(\varepsilon)$ be the time necessary to go from $(-\sqrt{\varepsilon},0)$ to the point $(\sqrt{\varepsilon},0)$ under the action of $P_f,$ i.e. $t_f(\varepsilon)$ is defined by
	$$P_f^{t_f(\varepsilon)}(-\sqrt{\varepsilon},0)=(\sqrt{\varepsilon},0).$$
\end{Def}

\begin{Def}
The curve $$\gamma(\varepsilon) := P_f^{\frac{1}{2}t_f(\varepsilon)}(-\sqrt{\varepsilon},0)$$ where $\varepsilon\geqslant0$ is called a \textit{bisector}.
\end{Def}

\begin{Lemma}
The bisector is smooth and transversal to the boundary $\{y=0\}.$
\end{Lemma}

\begin{figure}
\centering
\subfigure[Chart $(x,y).$]{
	\begin{tikzpicture}[scale=2.5]
	\filldraw (0,0) circle (0.2pt);
	\clip (-1,-0.4) rectangle (1.2,1.2);
	\draw[->, thick] (-1,0) -- (1,0) node[right] {$x$};
	\draw[->] (0,-1) -- (0,1) node[above] {$y$};
	\foreach \x in {-0.4, -0.3, -0.2, -0.1, 0.0, 0.1, 0.2, 0.3, 0.4, 0.5, 0.6, 0.7, 0.8, 0.9}
		\draw  (-1,\x-1) parabola bend (0,\x) (1,\x-1);
	\end{tikzpicture}}
\subfigure[Chart $(x,z).$]{
	\begin{tikzpicture}[scale=2.5]
	\filldraw (0,0) circle (0.2pt);
	\clip (-1,-0.4) rectangle (1.2,1.2);
	\draw[thick] (-1,1) parabola bend (0,0) (1,1);
	\foreach \x in {-0.4, -0.3, -0.2, -0.1, 0.1, 0.2, 0.3, 0.4, 0.5, 0.6, 0.7, 0.8, 0.9}
		\draw (-1,\x) -- (1,\x);
	\draw[->] (-1,0) -- (1,0) node[right] {$x$};
	\draw[->] (0,-1) -- (0,1) node[above] {$z$};
	\end{tikzpicture}}
\caption{Level sets of the function $f$ in charts $(x,y)$ and $(x,z)$. The thick curve is the boundary of $M.$}
\label{fig2}
\end{figure}
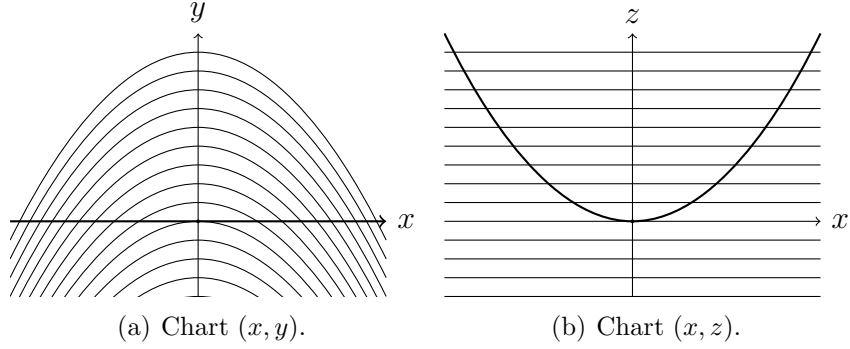

\begin{proof}
Let us introduce a new coordinate system $(x,z),$ where $$z(x,y):= f(x,y) = x^2+y$$ (see Figure \ref{fig2}).
Then in these new coordinates $f(x,z)=z$, $\omega = \omega(x,z)dx\wedge dz$, $y=0$ if and only if $z=x^2$, and $X_f=(-\frac{1}{\omega(x,z)},0)$.
Let us compute the function $t_f.$ Note that 
\begin{equation}\label{eqn:1}
-\omega(x,z)dx=dt 
\end{equation}
Integrating \eqref{eqn:1}  over the horizontal segment between the points $(-\sqrt{z},z)$ and $(\sqrt{z},z)$, we get 
$$t_f(z)=-\int\limits_{-\sqrt{z}}^{\sqrt{z}}\omega(\tau,z)d\tau.$$
In the same way we obtain equations for the bisector $(s(z),z)$

\begin{equation}\label{eqn:2}
    \int\limits_{-\sqrt{z}}^{s(z)}\omega(\tau,z)d\tau=\frac{1}{2}\int\limits_{-\sqrt{z}}^{\sqrt{z}}\omega(\tau,z)d\tau.
\end{equation}

Introducing a new variable $w=\sqrt{z}$ we obtain an equation for the function $\hat{s}(w):=s(w^2)$:

\begin{equation}\label{eqn:3}
    \int\limits_{-w}^{\hat{s}(w)}\omega(\tau, w^2)d\tau=\frac{1}{2}\int\limits_{-w}^{w}\omega(\tau, w^2)d\tau,
\end{equation}
Equation \eqref{eqn:3} allows us to define $\hat{s}(w)$ even if $w<0.$ We claim that $\hat{s}$ is a smooth function and $\hat{s}(-w) = \hat{s}(w).$ 

Partial derivative of \eqref{eqn:3}  with respect to $\hat{s}$ is $\omega(\hat{s}(w),w^2)$. For any $(x,z)$ we have $\omega(x,z)\ne0.$  
It follows from the implicit function theorem that $\hat{s}(w)$ depends smoothly on $w.$  

Now let us  let us make the following substitution in \eqref{eqn:3}: $w \rightarrow -w.$ We obtain:
\begin{multline*}
\int\limits_{-(-w)}^{\hat{s}(-w)}\omega(\tau, (-w)^2)d\tau=\frac{1}{2}\int\limits_{-(-w)}^{-w}\omega(\tau, (-w)^2)d\tau \\
\iff -\int\limits_{\hat{s}(-w)}^{w}\omega(\tau, w^2)d\tau=-\frac{1}{2}\int\limits_{-w}^{w}\omega(\tau, w^2)d\tau \\
\iff  \int\limits_{\hat{s}(-w)}^{w}\omega(\tau, w^2)d\tau= \frac{1}{2}\int\limits_{-w}^{w}\omega(\tau, w^2)d\tau \\
\iff \int\limits_{-w}^{\hat{s}(-w)}\omega(\tau, w^2)d\tau= \frac{1}{2}\int\limits_{-w}^{w}\omega(\tau, w^2)d\tau \iff \hat{s}(-w)=\hat{s}(w).
\end{multline*}

It means that equation \eqref{eqn:3} defines $\hat{s}$ as an even function of $w$.
	$$s(z)=s(\sqrt{z}^2)=\hat{s}(\sqrt{z}),$$
so $s$ is a smooth function of $z.$ Now it is clear that the bisector is transversal to the boundary $\{z=x^2\}.$
\end{proof}

\begin{Rem}
It follows from the proof of Lemma~$3$ that the bisector can be smoothly extended to the lower half plane.
\end{Rem}

\begin{Def}
Let $T_f(x,y)$ be be the time necessary to go from the bisector to the point $(x,y)$ under the action of $P_f$. 
\end{Def}

\begin{Rem}
In the chart $(x,z)$, we have:
$$T_f(x,z)=\int\limits_{s(z)}^{x}-\omega(\tau,z)d\tau = \int\limits_{x}^{s(z)}\omega(\tau,z)d\tau,$$ 
where the function $s$ is defined in Lemma $3$.
Now it is clear that $T_f$ is a smooth function. 

Also note that since $P_f$ is the flow of the vector field $X_f$, it follows that $dT_f(X_f)=1.$
\end{Rem}

\begin{Lemma}
 	$\omega= df \wedge dT_f.$  
\end{Lemma}

\begin{proof}
	Using that $dT_f(X_f)=1$ (see Remark 4), we get
	$$i_{X_f} df \wedge dT_f = df(X_f)dT_f - df dT_f(X_f)=  -dT_f(X_f) df = - df = i_{X_f} \omega,$$
	so $$i_{X_f} (df \wedge dT_f - \omega) = 0,$$
	and, since the ambient surface is 2-dimensional and $X_f \neq 0$, it follows that $\omega = df \wedge dT_f.$ 
\end{proof}

\begin{Lemma}
	$\frac{d}{d\varepsilon}A_f(\varepsilon)=|t_f(\varepsilon)|.$
\end{Lemma}

\begin{proof}
	To proof this, let us use the chart $(x,z)$ from Lemma~$3.$ Remind that in this chart $f(x,z)=z.$ Now it follows from the definition of $A_f$ and from Lemma~$4$ that 
	\begin{multline*}
	A_f(\varepsilon + \delta) - A_f(\varepsilon) =| \int\limits_{-\sqrt{\varepsilon}}^{\sqrt{\varepsilon}}\int\limits_\varepsilon^{\varepsilon+\delta}dz \wedge dT_f| + o(\delta)=
	|\int\limits_{-\sqrt{\varepsilon}}^{\sqrt{\varepsilon}} dT_f \int\limits_\varepsilon^{\varepsilon+\delta} dz |+ o(\delta) = \\
	= \delta |T_f(\sqrt{\varepsilon},0)-T_f(-\sqrt{\varepsilon},0)|+ o(\delta) = \delta |t_f(\varepsilon)| + o(\delta).
	\end{multline*}
	So $$\frac{d}{d\varepsilon}A_f(\varepsilon)=|t_f(\varepsilon)|.$$
	
\end{proof}

\begin{Lemma}
 	Suppose that after a coordinate transformation $(x,y) \rightarrow (p,q)$ the following conditions hold:
\begin{enumerate}
\item
$f(p,q) = p^2 + q.$
\item
$\omega = dp \wedge dq.$
\item
The equation $p=0$ describes the bisector.
\item
$A_f(\varepsilon)=\frac{4}{3}\varepsilon \sqrt{\varepsilon}.$

\end{enumerate}
	Then $y(p,q)=0$ if and only if $q=0$.
\end{Lemma}

\begin{proof}
First of all, by the Condition~$4$ the function $A'_f(\varepsilon)$ can be computed as:
\begin{equation}\label{eqn4}
\frac{d}{d\varepsilon}4/3\varepsilon \sqrt{\varepsilon}=2\sqrt{\varepsilon}.
\end{equation}

Let us check that $y=0$ if and only if $q=0.$ It is follows from Lemma~$3$ that the curve $\{y=0\}$ is transversal to the bisector $\{p=0\}$. So, the curve $y = 0$ is a graph of some function $q = r(p)$ (see Figure~$3$).
It follows from the definition of bisector that $r(x)=r(-x).$
Let us proof that $r(x)\equiv0$.  Assume that there exists some $p_0$ such that $q_0 := r(p_0)>0$ (the case $q_0<0$ is analogous).  
\begin{multline*}
A'_f(q_0+p_0^2) = [ \text{by equation~(\ref{eqn4})} ] = 2\sqrt{q_0+p_0^2} > |2 p_0|  \\
= [\text{by conditions (1),(2),(3) and the definition of }t_f] = |t_f(q_0+p_0^2)| \\
= [\text{by Lemma 5}] = A'_f(q_0+p_0^2).
\end{multline*}
This contradiction concludes the proof.

\end{proof}

\section{Proof of the main theorem}

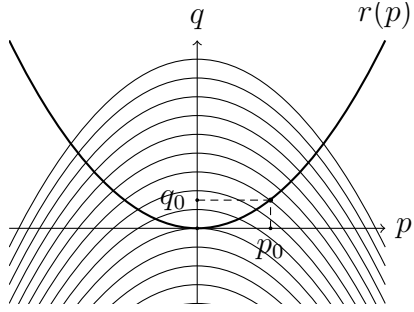
\begin{figure}
\label{fig3}
\centering
	\begin{tikzpicture}[scale=2.5]
	\filldraw (0,0) circle (0.2pt);
	\filldraw (0.39,0) circle (0.2pt) node[below] {$p_0$};
	\filldraw (0,0.15) circle (0.2pt) node[left] {$q_0$};
	\filldraw (0.39,0.15) circle (0.3pt);
	\clip (-1,-0.4) rectangle (1.2,1.2);
	\draw (-1,0) -- (0,0);
	\draw[densely dashed] (0.39,0) -- (0.39,0.15);
	\draw[densely dashed] (0,0.15) -- (0.39,0.15);
	\draw[->] (0,0) --  (1,0) node[right] {$p$};
	\draw[->] (0,-1) -- (0,1) node[above] {$q$};
	\draw[thick] (-1,1) parabola bend (0,0) (1,1) node[above] {$r(p)$};
		\foreach \x in {-0.4, -0.3, -0.2, -0.1, 0.0, 0.1, 0.2, 0.3, 0.4, 0.5, 0.6, 0.7, 0.8, 0.9}
	\draw  (-1,\x-1) parabola bend (0,\x) (1,\x-1);
	\end{tikzpicture}
\caption{Illustration to the proof of Lemma~5.}
\end{figure}

\begin{proof}

Consider the function  
$$\alpha(\varepsilon):=A_f^{-1}(\frac{4}{3}\varepsilon \sqrt{\varepsilon}).$$
It follows from Lemma 2 that  $\alpha$ is a smooth function.
Let also 
$$H(x,y) := [\alpha^{-1}\circ f](x,y)$$
$$p(x,y):=-T_H(x,y)$$ 
$$q(x,y):=H-p^2(x,y).$$
Then
\begin{multline*}
dp\wedge dq=-dT_H\wedge d(H-T_H^2)=\\
=-dT_H\wedge dH + dT_H\wedge 2T_H dT_H = dH\wedge dT_H= [\mbox{by Lemma 4}] = \omega,
\end{multline*}
so $dp$ and $dq$ are linearly independent. 
Further, in the chart $(p,q)$, we have

\begin{enumerate}
\item
$H(p,q)=p^2+q$ and $f(p,q)=\alpha(p^2+q).$
\item
$\omega=dp\wedge dq.$
\item
The equation $p=0$ describes the bisector, because $p(x,y)=0$ if and only if $T(x,y)=0,$ while the latter means that the point $(x,y)$ belongs to the bisector.
\item
$A_H(\varepsilon)=A_f(\alpha(\varepsilon))=A_f(A_f^{-1}(\frac{4}{3}\varepsilon \sqrt{\varepsilon}))= \frac{4}{3}\varepsilon \sqrt{\varepsilon}.$
\end{enumerate}

So the chart $(p, q)$ fulfils all conditions of Lemma~$6$.  And now it follows from Lemma~$6$ that the chart $(p,q)$ satisfies all conditions of Theorem~$1'.$

\end{proof}

\section{Acknowledgements} 
The author is grateful to A.M. Izosimov and A.A. Oshemkov for useful comments and discussions. This research is supported in part by the Russian Foundation for Basic Research (grant No. 16-01-00378-a), the program ‘‘Leading Scientific Schools’’ (grant no. NSh-6399.2018.1) and the Simons Foundation.

\bibliography{bibl}{} 
\end{document}